\documentclass{article}

\usepackage{hyperref}
\usepackage{amssymb, latexsym, amsthm, amsmath, verbatim,graphicx}
\newtheorem{definition}{Definition}
\newtheorem*{definition*}{Definition}
\newtheorem{theorem}{Theorem}
\newtheorem{prop}[theorem]{Proposition}
\newtheorem*{theorem*}{Theorem}
\newtheorem*{prop*}{Proposition}
\newtheorem*{lemma*}{Lemma}

\newtheorem*{algorithm*}{Algorithm}
\begin{document}

\title{%
A Survey of Adwords Problem With Small Bids In a Primal-dual Setting\\
\large Greedy Algorithm, Ranking Algorithm and Primal-dual Training-based Algorithm}
\author{Haoqian Li}
\date{\today}

\maketitle
\begin{abstract}
The Adwords problem has always been an interesting internet advertising problem. There are many ways to solve the Adwords problem with adversarial order model, including the Greedy Algorithm, the Balance Algorithm and the Scale-bid Algorithm, which is also known as MSVV. In this survey, I will review the Adwords problem with different input models under a primal-dual setting, and with the small-bid assumption. In the first section, I will focus on Adwords with adversarial order model, and use duality to prove the efficiency of the Greedy Algorithm and the MSVV algorithm. Next, I will look at a primal-dual training-based algorithm for the Adwords problem with the IID model. 
\end{abstract}
\section{Introduction}
Nowadays, almost all of the technology companies, such as Google and Yahoo, need to face a combinatorial problem of assigning user queries to advertisers to maximize the total revenue. Specifically, a search engine might receive offers from $n$ different companies: each month, company $u$ wants to pay at most $B_u$ dollars in advertising; Company $u$ is willing to pay $w_{uv}$ if its ad appears when word $v$ is searched. The goal of the search engine is to maximize the profit of the search engine, which equals to the total amount of money spent by advertisers. \\
\subsection{Combinatorial Formulation} Mathematically, we are given a set of $n$ nodes $U$, with each $u\in U$ having budget $B_u$. At each time $i=1,\ldots, t$: A node $v_i$ arrives, bids $w_{uv_i}$ are displayed, and we must decide either to match $v_i$ to $u\in U$ and earn min($w_{uv_i}, RB_u$), or not to match $v_i$ to any node (i.e., no ad is shown when $v_i$ is searched). $RB_u$ stands for the remaining budget for node $u$. We aim to maximize the total amount of money spent by advertiser.
\subsection{Duality Review}
Before I dive into the IP formulation of the problem, I first give a quick review of primal-dual formulation and duality theorem.\\
For every LP in the form \begin{align*}\text{max } &c^Tx\\
\text{s.t. } &Ax\leqslant b\\
&x\geqslant 0,
\end{align*}call this the primal problem, there always exist a dual problem associate to it in the following form 
\begin{align*}\text{min } &b^Ty\\
\text{s.t. } &A^Ty\geqslant c\\
&y\geqslant 0.
\end{align*}
We have also seen in class that $c^Tx\leqslant b^Ty$ for any feasible solution of the primal $x$, and any feasible solution of the dual $y$. In other words, the value of any feasible solution to the dual yields an upper bound on the value of any feasible solution to the primal. In addition, if $\bar{x}$ is the optimal solution to the primal and $\bar{y}$ is the optimal solution to the dual, we have $c^T\bar{x}=b^T\bar{y}$ (the costs/object value coincide). \\

Now, what remains to be crucially important in this paper is the complementary slackness: Let $s=b-Ax\geqslant 0, w=A^Ty-c \geqslant 0$ be the nonnegative primal and dual slacks. Then $b^T \bar{y}=c^T\bar{x}$ (obtain optimal solution), iff $y^Ts=0$ and $w^Tx=0$. In other words, if $\bar{x_i}>0$ for some $i$, then the corresponding dual slack is 0 (constraint is tight), i.e., $A_i^T\bar{y}=c_i$, and if the primal slack is nonzero (constraint not tight), i.e., $A_i^T\bar{x}<b_i$, then corresponding dual variable $\bar{y_i}=0$. Similar logic follows for dual variable and dual slack.
\subsection{Primal-dual Formulation}Buchbinder et al.\cite{Buch} studied the Adwords problem with the small bids assumption using the Primal-Dual approach. For now we assume the Adwords problem is offline, which means that the complete bipartite graph is available to us and we can solve the optimal solution. We first look at the primal problem. The formulation is as follows:\\
\begin{align*} \text{max} \displaystyle\sum_{u,v} x_{uv} w_{uv} &    &\text{max  } &w^T x \\
s.t. \forall u: \displaystyle\sum_v x_{uv} w_{uv} & \leqslant B_u   &s.t. &Ax\leqslant b\\
\forall v: \displaystyle\sum_u x_{uv} &\leqslant 1 &A=
 \begin{bmatrix}
  w^T\\
  e^T \\
 \end{bmatrix}   &b=
 \begin{bmatrix}
  B_u\\
  e \\
 \end{bmatrix}  
\\
x_{uv}&\geqslant 0
\end{align*}
Note $x_{uv}=1$ if we match $v\in V$ to $u\in U$ and earn $w_{uv}$, and $x_{uv}=0$ otherwise. The first constraint states that the total amount company $u$ pays cannot exceed its budget, and the second constraint says that each person can only be shown to at most one ad. The solution could be fractional but rounding-off lose very little in the object value.\\
We can construct the dual problem:\\
\begin{align*}
\text{min} \displaystyle\sum_{u} B_u\alpha_u+ \displaystyle\sum_{v} \beta_v&  &\text{min  } &b^T y\\
s.t. \forall u,v: w_{uv} \alpha_u+\beta_v &\geqslant w_{uv}   &s.t. &A^T y\geqslant w\\
\alpha_u&\geqslant 0 \text{ } \forall u\\
\beta_v&\geqslant 0 \text{ } \forall v
\end{align*}
\textbf{Remark}: by complementary slackness, we know that (1): If $x_i>0$ for some i, then the corresponding dual constraint is satisfied at equality, i.e., $A_i^T \bar{y}=w_i$; (2): If $y_i>0$ for some $i$, then the corresponding primal constraint is satisfied at equality, i.e., $A_i \bar{x}=b_i$; (3) If $A_i\bar{x}<b_i$, then $y_i=0$. Therefore the optimal primal and dual solutions $\{x_{uv},\alpha_u,\beta_v\}$ must satisfy the following conditions:\\
\begin{align}x_{uv}>0 &\rightarrow w_{uv}(1-\alpha_u)=\beta_v\\
\alpha_u>0 &\rightarrow \displaystyle\sum_v x_{uv}w_{uv}=B_u \\
\displaystyle\sum_v x_{uv}w_{uv}<B_u &\rightarrow \alpha_u=0\\
\beta_u>0 &\rightarrow \displaystyle\sum_u x_{uv}=1 \nonumber
\end{align}
Since the first constraint of the primal problem gives us $w_{uv}(1-\alpha_u)\leqslant \beta_v$, comparing to equation (1), we can conclude that\\
\textbf{Rule (1): $v$ is allocated (ie., $x_{uv}>0$) to the bidder who maximizes the scaled bid $w_{uv}(1-\alpha_u)$}.\\ In other words, given only the optimal dual variables $\alpha_u$, one can reconstruct the optimal primal solution by assigning each $v$ to the bidder $u$ who maximizes $w_{uv}(1-\alpha_u)$. Equation (2) and (3) tells us that \\
\textbf{Rule 2: $\alpha_u$ is positive only for bidders $u$ who have finished their budget and remains to be 0 if budget not finished}.\\ In the online Primal-dual problem, one can utilize these two rules, i.e., maintain a best estimate for the optimal dual variables $\alpha_u$ and use the first rule when deciding which $u$ to assign.
\section{Online Greedy Algorithm and MSVV Algorithm}
\subsection{Greedy Algorithm}
I first provide the Greedy Algorithm for the online Adwords problem:\\
\begin{align*}&\textbf{Initialize}: \alpha_u=0 \text{ }  \forall u, \beta_v=0 \text{ }\forall v, x_{uv}=0\text{ } \forall u,v.\\
&\textbf{When the next vertex }  v\in V \textbf{arrives:}\\
& \text{If v has no available neighbors, continue}\\
&  \text{Match v to that available neighbor } u^* \text{ which maximizes } bid_{uv}\\
&\textbf{Update: } \alpha_u^*=\begin{cases}
1 & \text{if }u^* \text{ has finished its budget}\\
0 & \text{if }u^* \text{ has not yet finished its budget}\end{cases}\\
&\beta_v=bid_{u^{*}v}\geqslant bid_{uv}\\
&x_{u^*v}=1
\end{align*}Note that the because of the small-bid assumptions, when each $v$ arrives, we can assume that $bid_{uv}=\text{min}(w_{uv}, RB_u)=w_{uv}$, i.e., taking $bid_{uv}$ will never exceed the remaining budget of $u$. Also note that since every available vertex $u$ has $\alpha_u=0$ (hasn't finished its budget), it is the same as saying  "match $v$ to the available neighbor $u^*$ which maximizes $bid_{uv}(1-\alpha_u)$."
\begin{theorem}Greedy achieves a ratio of $\frac{1}{2}$ for the Adwords problem with the small-bids assumption.
\end{theorem}I prove this theorem by showing that at the end of the algorithm, the solution is feasible for both the primal and the dual problem, and the primal objective value is at least one half of the optimal value. I start by proving a proposition:
\begin{prop}For the Adwords problem with the small-bids assumption, if at the end of the algorithm, the primal objective is at least $\frac{1}{k}$ times the dual objective, then the algorithm achieves a ratio of $\frac{1}{k}$ of the optimal solution.
\end{prop}
\begin{proof}Define $Primal$ as the primal objective value, $Dual$ as the dual objective value, $Dual^*$ as the optimal dual value, and $OPT$ as the optimal objective value. We have the following:\begin{center} $Primal\geqslant \frac{1}{k} Dual \geqslant \frac{1}{k} Dual^* = \frac{1}{k} OPT$.
\end{center}
The first inequality is given. Since the dual problem is minimization, we must have $Dual \geqslant Dual^*$ for $Dual^*$ to be optimal. By strong duality, we also have $Dual^*=OPT$.
\end{proof}
We now prove Theorem 1:
\begin{proof}\textbf{(1) Feasibility:} I first prove the algorithm outputs a feasible solution through demonstrating that the primal and the dual problems are both feasible. First the primal solution is feasible since we only allocate $v$ to the bidder $u$ who has not finished its budget yet, and allocate $v$ to only one bidder. For the dual problem constraint, at the end of the algorithm, if $u$ has finished its budget, we know $\alpha_u=1$ and clearly $w_{uv}\alpha_u+\beta_v\geqslant w_{uv}$. If $u$ still has budget remaining, we need to show $\beta_v\geqslant w_{uv}$. However, at each iteration as we asign $v$, we pick $u$ such that $u^*$ maximizes $w_{u^*v}$ and update $\beta_v=w_{u^*v}$. Therefore, $\beta_v = w_{u^*v}\geqslant w_{uv}$.\\
\textbf{(2) Ratio:}By proposition 2, we need to show that the objective is at least $\frac{1}{2}$ of the dual objective. At each iteration when $v$ shows up, we allocate it to the optimal $u^*$, and two cases might occur: $u^*$ has finished its budget by spending $w_{u^*v}$, or $u^*$ still has budget available. In the latter case, the primal objective increases by $x_{u^*v}w_{u^*v}=w_{u^*v}$, and the dual objective also increases by $\beta_v=w_{u^*v}$. In the former case, both objective increase by $w_{u^*v}$, and in the dual objective function, $\alpha_u^*$ also increases from 0 to 1. Therefore we add $B_{u^*}=\displaystyle\sum_{v\in S} w_{u^*v}$ where $S$ is the set of all $v's$ asigned to $u^*$. Thus we count all of the $w_{u^*v}$ again for all of the previously $v$ allocated to $u$. The worst case being that all $u's$ have finished their budget and the dual objective value is twice as the primal one. In other words, the primal objective is at least $\frac{1}{2}$ times the dual objective.
\end{proof}
\subsection{Online MSVV Algorithm}
In the greedy algorithm, we set the dual $\alpha_u$ to 0 or 1. Similarly, we can find the best online dual variables $\alpha_u$ as a function of the fraction of budget spent in the MSVV algorithm. For simplicity we take all budget $B_u=1, \forall u$. Define \begin{center}
$k=\frac{1}{1-\frac{1}{e}}$\end{center}
\begin{center}$x_u$=fraction of budget of $u$ that has been spent so far\end{center}
 \begin{center}$f(x)=\frac{e^x-1}{e-1}, \forall x\in [0,1]$\end{center}
 \begin{center}$\Delta_{uv}(x)=f'(x)w_{uv}=ke^{x-1}w_{uv}$\end{center}
I now provide the online MSVV algorithm:\begin{align*}&\textbf{Initialize}: \alpha_u=0 \text{ }  \forall u, \beta_v=0 \text{ }\forall v\\
&\textbf{When the next vertex }  v\in V \textbf{arrives:}\\
& \text{If v has no available neighbors, continue}\\
&  \text{Match v to that available neighbor } u^* \text{ which maximizes } kw_{uv}-\Delta_{uv}(x_u)\\
&\textbf{Update: } \alpha_u=\alpha_u+\Delta_{uv}(x_u^*)\\
&\beta_v=kw_{uv}-\Delta_{uv}(x_u^*)\\
\end{align*}
\textbf{Remark:} Note that while deciding which $u$ we assign $v$ to, we have \begin{center}$kw_{uv}-\Delta_{uv}(x_u)=kw_{uv}(1-e^{x_u-1})=kw_{uv}(1-e^{-(1-x_u)})$,\end{center} where $(1-x_u)=B_u-x_u=RB_u$ stands for the remaning budget of $u$. Therefore it is equivalent to assign the incoming node $v$ to a neighbor $u$ that maximizes $w_{uv} f(\frac{RB_u}{B_u})=w_{uv} f(RB_u)$, which we have seen in class.
\begin{theorem}Online MSVV algorithm achieves ratio $\frac{1}{k}=1-\frac{1}{e}$ for the Adwords problem with the small bids assumption.
\end{theorem}Again I prove the theorem by showing feasiblity and the primal-dual ratio. I first prove another proposition:
\begin{prop}At any point in time, $\alpha_u=f(x_u), \forall u\in U$.
\end{prop}
\begin{proof}WLOG, given $u\in U$, suppose at any given time $t$, $v_1,\ldots,v_k$ were allocated to $u$ so far in that order, and let $x_1,\ldots, x_k$ be the spending of $u$ up until the times when $v_1,\ldots, v_k$ arrives. Note that the cumulative spending of $u$ as $v_i$ arrives equals to the sum of all previous costs/bids ($w_{uv_j}$), where $j=1,\ldots, i-1$. Thus we have $x_i=\sum_{j=1}^{i-1}w_{uv_j}$. According to our algorithm, at the time $v_k$ has been assigned to $u$, $\alpha_u=\sum_{i=1}^{k} \Delta_{uv_i}(x_i)$. Thus we have \begin{center}$\alpha_u=\sum_{i=1}^{k} \Delta_{uv_i}(\sum_{j=1}^{i-1}w_{uv_j})=\int_{x=0}^{x_u} ke^{x-1}dx=k(e^{x_u-1}-e^{-1})=\frac{e^{x_u}-1}{e-1}$
\end{center}We have the second equality because we assume that every piece of $w_{uv_i}$ is infinitely small, and thus can repleace it by an integral.
\end{proof}
\begin{proof}\textbf{(1) Feasibility: }First, the primal solution is feasible since we only assign $v$ to $u$ who has not finished its budget. Now I show the feasibility for the dual problem, i.e., at the end of algorithm: $w_{uv} \alpha_u+\beta_v \geqslant w_{uv}$, or equivalently $\beta_v \geqslant w_{uv}(1-\alpha_u)$. When an arbitrary $v$ arrives, suppose we assign it to $u^*$, and let $x_{u^*}$ be the fraction of budgets spent by $u^*$. Let $u$ be some other bidders that is not optimal and let $x_u$ be the fraction of budgets spent by $u$. Note $k w_{u^*v}-\Delta_{u^*v}(x_{u^*})\geqslant k w_{uv}-\Delta_{uv}(x_u)$ since in each iteration we pick $u$ that maximizes the value of the expression. Let $X_u$ be the fraction of budget spent by $u$ at the end of the algorithm $(x_u\leqslant X_u)$. We have \begin{align*}\beta_v &=k w_{u^*v}-\Delta_{u^*v}(x_{u^*})  \text{              (by definition of }\beta)\\
&\geqslant k w_{uv}-\Delta_{uv}(x_u) \\
&= w_{uv}(k-ke^{x_u-1})=\frac{1-e^{x_u-1}}{e-1}\\
&=w_{uv}(1-f(x_u))\geqslant w_{uv}(1-f(X_u)  \text{              (by monotonicity of }f(x))\\
&=w_{uv}(1-\alpha_u)
\end{align*}
\textbf{(2) Primal-Dual Ratio:} In each iteration, the primal objective increases by $w_{uv}$, and the dual objective increases by $\alpha_u+\beta_v=\Delta_{uv}(x_u)+kw_{uv}-\Delta_{uv}(x_u)=kw_{uv}$. Therefore at the ned of the algorithm, the primal objective is at least $\frac{1}{k}$ times the dual objective, by Proposision 2, the proof is complete.
\end{proof}
\section{Primal-dual Training-based Algorithm}
In this section, I first introduce the general class of packing linear programs (PLP), and then introduce the primal-dual training-based algorithm and prove that it achieves "$1-O(\epsilon)$" approximation for the online stochastic PLP problem with high probability with mild assumptions. The "$1-O(\epsilon)$" approximation/competitive means that with high probability under the randomness in the stochastic model, the algorithm achieves at least $1-O(\epsilon)$ fraction of the objective value of the offline optimal solution ($OPT$) for the actual instance.
\subsection{The General Class of PLP}
Let $J$ be a set of $m$ resources, and each resource $j\in J$ has a capacity $c_j$. The set of resources and their capacities are known in advance. Let $I$ be a set of $n$ agents that arrive one by one online, each with a set of options $O_i$. Each option $o\in O_i$ of agent $i$ has an associated value $w_{io}\geqslant 0$ and requires $a_{ioj}$ units of each resource $j$. The set of options and the values $w_{io}$ and $a_{ioj}$ arrive together with agent $i$. When an agent arrives, the algorithm has to immdediately decide whether to assign the agent and which option to choose. The goal is to find a $maximum-value$ allocation that does not allocate more of any resource than is available.\\
\textbf{Remark 1:} Comparing to the AdWords problem, we can see that $J=U$ stands for the set of companies/advertisers, each of which has budget $c_j=B_u$ available. $I=V$ is the set of nodes arrives. As $v_i$ arrives, it has a set of options $O_i=\{\text{available neighbors } u\}$. Each option has an associated value $w_{io}=a_{ioj}=w_{ij}$.\\
\textbf{Remark 2:} Since PLP is a generalization of AdWords problem, where the key difference being that $w_{io}$ and $a_{ioj}$ are unrelated in PLP, this algorithm also applies to AdWords problem, for which the proof is even simpler.\\
We only need to adjust the formulation in Section 1.2 a little bit to adapt the PLP. For convenience, we manipulate the first constraint of the primal: we divide both sides by $c_j(=B_u)$, and the right-hand side becomes 1, the left-hand side becomes $a_{ioj}^*=a_{ioj}/c_j$. We use notation $a_{ioj}$ representing $a_{ioj}^*$ below for simplicity. Also note that $i\equiv v$ and $\{o\in O_i\}\equiv u$. The primal and dual follows:
\begin{align*} \text{max} \displaystyle\sum_{i}\displaystyle\sum_{o\in O_i} x_{io} w_{io} &   &\text{min }  \displaystyle\sum_{j}\alpha_j+ \displaystyle\sum_{i} \beta_i\\
\forall j: \displaystyle\sum_{i} x_{io} a_{ioj} & \leqslant 1   &\forall i,o: a_{ioj} \alpha_j+\beta_i &\geqslant w_{io}\\   
\forall v: \displaystyle\sum_{o\in O_i} x_{io} &\leqslant 1   &\alpha_j, \beta_i \geqslant 0\\
x_{io}&\geqslant 0
\end{align*}
\subsection{Training-Based Primal-Dual Algorithm (PTAS)}
We define $n$ to be the total number of agents, $m$ as the number of constraints, $q=\text{max}_i |O_i|$ as the maximum number of options for any agent. Finally, define the $gain$ from option $o\in O_i$ as $w_{io}- \sum_j \alpha_j a_{ioj}$. Recall in the MSVV Algorithm, we match $v$ to $u$ which maximizes $\beta_v=kw_uv-\Delta_{uv}(x_u)=kw_{uv}(1-e^{-(1-x_u)})=w_{uv}f(RB_u)$. In other words, $\beta_v$ is the maximum $"gain"$ we obtain as each $v$ comes. Also recall that $\alpha_u=f(x_u)\forall u\in U$. Therefore, the $gain$ in PLP, if applied to Adwords Problem, is $w_{uv}-\sum_u \alpha_u w_{uv}=w_{uv}(1-\sum_u \alpha_u)=w_{uv}(B_u-\sum_u f(x_u))=w_{uv}f(RB_u)$ (we have already normalized $B_u=1$). We can see the nice analogy between $gain$, $\beta_v$, and what we have seen in class.\\

The main idea behind the algorithm is to solve the LP on a sample of the queries. Note that we cannot expect to get a representative sample of all types of agents from a small sample. In other words, we cannot expect to estimate the ditribution of $w_{io}$ and $a_{ioj}$. We then use only the dual variables $\alpha_j$ and $\beta_i$ from the sampled LP to guide us for solving the rest of query stream.\\
\textbf{Algorithm:}\\
1. Let $S$ denote the first $\epsilon n$ agents in the sequence. We do not select these agents for the purpose of analysis. May assign them according to some online algorithm. Let $\hat{D}$ denote the sampled version of the dual program on the agents in S with following changes: for all $j\in [m]$, replace $c_j$ by $\epsilon c_j$, which is equivalent to reducing the capacity of a constraint from 1 to $\epsilon$ (i.e., change the right-hand side of the first constraint to $\epsilon$).\\
2. Solve $\hat{D}$, let $\alpha_j^*$ denote the value of the dual variable for constraint $j$ in this optimal solution.\\
3. For each subsequent agent $i$, select the option $o$ that provides maximum gain, and set $\beta_i=gain(o)$. Let $w_{max}=\text{max}_{i,o}\{{w_{io}}\}$, $a_{max}=\text{max}_{i,o,j}\{{a_{ioj}}\}.$
Note $\alpha_j^*$ serves as a cost/price per unit of the resource for the remaining agents.
\subsection{Theorem and Proof}
\begin{theorem}The Training-Based Primal-Dual algorithm is $(1-O(\epsilon))$-competitive for the online stochastic PLP problem with high probability, as long as folloing conditions hold: (1) $\text{max}_{i,o}\{\frac{w_{io}}{OPT}\}\leqslant \frac{\epsilon}{(m+1)(ln n+ln q)}$ and (2) $\text{max}_{i,o}\{\frac{a_{ioj}}{c_j}\}\leqslant \frac{\epsilon^3}{(m+1)(ln n+ln q)}$.
\end{theorem}
\textbf{Remark: }Note that the first condition means that no individual option provides too large a fraction of the total value, and the second condition is equivalent to say that no individual option for any agent consumes too much of any resource.\\

The key idea behind the proof is that if $\alpha_i^*$ satisfies the complementary slackness condition on the first $\epsilon n$ agents (being an optimal solution for our sample), then w.h.p it approximately satisfies these conditions on the entire set. I begin the proof by introducing some key definitions.
\begin{definition}Let $I^*\subseteq I$ denote the set of agents $i$ with some option $o$ having non-negative gain, i.e., the set of agents that will be allocated to some resource $j$ (we do not assign any agent $i$ that provides negative gain). Let $O^*$ denote the set of pairs $\{(i,o)| i\in I^*, o=arg\text{max}_{o\in O(i)} gain(o)\}$, i.e., the set of best possible allocations between all agents $i\in I$ and all resources $j\in J$. Let $O^*(S)=O^*\cap S,$ the best possible allocation for the sampled queries. Consequently, we have $O^*-O^*(S)$ represents the allocation of options selected by our algorithm.\end{definition}\textbf{Remark: } For the purpose of analysis, our algorithm do not select any options for agents in $S$. In our algorithm, given a vector $\alpha^*$, by selecting for each agent in $I^*$, the option $o$ that maximizes $gain(o)$ for $\{i\in I^*, o\in O(i)\}$, i.e., $(i,o)\in O^*$ and set $\beta_i=gain(o)$, we obtain a feasible solution to the \textbf{Dual-LP}. Because of the analogy between $\beta_v$ in MSVV and $\beta_i=gain(o)$, the proof is analogous to the feasibility proof for theorem 3 and is omitted here.
\begin{definition}Let $W=\sum_{(i,o)\in O^*}w_{io}$ be the total weight of selected options, and $W(S)=\sum_{(i,o)\in O^*(S)}w_{io}$ be the total weight of selected options of the sample. Let $C(j)=\sum_{(i,o)\in O^*}a_{ioj}$ be the total consumed resources for $j\in J$ and similarly $C(j,S)=\sum_{(i,o)\in O^*(S)}a_{ioj}$.
\end{definition}
For any fixed vector $\alpha^*$, $O^*, W, $ and each $C(j)$ are independent of the choice of the sample S. Thus the expected value of $W(S)$ is $\epsilon W$, and that of $C(j,S)$ is $\epsilon C(j)$. Also note that after implementing some online algorithm for the sampled queries, the result  should be close to $\epsilon W$ and $\epsilon C(j)$.	
\begin{definition}For a sample $S$ and $j\in [m]$, let $r_j(S)=|C(j,S)-\epsilon C(j)|$, and let $t(S)=|W(S)-\epsilon W|$. When refering to the sample, we abbreviate $r_j(S)$ as $r_j$ and $t(S)$ as $t$.\\

1. The sample S is $r_j-{bad}$ if: \begin{center}$r_j\geqslant (m+1)(\text{ln} n+\text{ln} q)a_{max}+\sqrt{C(j)} (2\sqrt{\epsilon(m+1)(\text{ln}(n)+\text{ln}q)a_{max}})$.
\end{center}

2. The sample $S$ is $t-bad$ if:
\begin{center}$t\geqslant (m+1)(\text{ln} n+\text{ln} q)w_{max}+\sqrt{W} (2\sqrt{\epsilon(m+1)(\text{ln}(n)+\text{ln}q)w_{max}})$
\end{center}
\end{definition}

Now I prove Theorem 5, our central theorem. For the following parts, we assume the condtions of Theorem 5 hold. We prove Theorem 5 by proving the following two propositions:
\begin{prop}If the sample S is not t-bad or $r_j-bad$ for any constraint $j$, the value of the options selected by the algorithm is $(1-O(\epsilon)) OPT$.
\end{prop}
\begin{prop}The sample $S$ has high probability to be good for any fixed $\alpha^*$, i.e., not t-bad or $r_j$-bad for any $j$.\end{prop} Therefore, the combination of the two propositions tell us that, with high probability, our algorithm returns a feasible solution with value at least $(1-O(\epsilon))OPT$, completing the proof of Theorem 5.\\

I first prove Proposition 6 by first proving another lemma.
\begin{lemma*}Let $j\in [m]$ be a constraint such that $C(j,S)=\epsilon$. If $S$ is not $r_j-bad$, under constraints of theorem 5, we have $1-2\epsilon \leqslant C(j) \leqslant 1+3(\epsilon+\epsilon^2)$.
\end{lemma*}
\begin{proof}Given $S$ is not $r_j-bad$, we have $|C(j,S)-\epsilon C(j)|\leqslant (m+1)(\text{ln} n+\text{ln} q)a_{max}+\sqrt{C(j)} (2\sqrt{\epsilon(m+1)(\text{ln}(n)+\text{ln}q)a_{max}})$. Consequently $C(j,S)-\epsilon C(j) \leqslant (m+1)(\text{ln} nq)a_{max}+\sqrt{C(j)} (2\sqrt{\epsilon(m+1)(\text{ln}(nq)a_{max}})$. By our assumption in theorem 5, we also have $a_{max}\leqslant \epsilon^3/((m+1)(\text{ln} nq))$. Hence by substituing $a_{max}$, we have \begin{center}$C(j,S)-\epsilon C(j) =\epsilon-\epsilon C(j) \leqslant \epsilon^3+\sqrt{C(j)} (2\sqrt{\epsilon \epsilon^3})=\epsilon^3+2\sqrt{C(j)}\epsilon^2$\end{center}
Then\begin{align*}1-C(j)&\leqslant \epsilon^2+2\epsilon \sqrt{C(j)}\\
C(j)+2\epsilon \sqrt{C(j)} + \epsilon^2&\geqslant 1\\
(\sqrt{C(j)})+\epsilon)^2 &\geqslant 1\\
\sqrt{C(j)})+\epsilon &\geqslant 1\\
C & \geqslant 1+\epsilon^2-2\epsilon \geqslant 1-2\epsilon
\end{align*}
The upper bound proof is similar.
\end{proof}
Now I prove Proposition 6:\begin{proof}Let $D=\sum_j \alpha_j^*+\sum_{i\in O^*}\beta_i$ be the value of the feasible dual solution obtained by setting $\beta_i=gain(o)$ for each $(i,o)\in O^*$. By weak duality, the dual objective value $D$ serves as an upper bound on the optimal solution $OPT$. Showing that $\sum_{(i,o)\in O^*-O^*(S)}w_{io}\geqslant (1-O(\epsilon))D\geqslant (1-O(\epsilon)) OPT$ suffices to prove the proposition. Let $J_1$ denote the set of constraints $j\in m$ such that $\alpha_j^*>0$, and $J_2=[m]-J_1$ be the set of constraints such that $\alpha_j^*=0$. Propsition 6 implis that if for some $j\in m$, the capacity of $j$ has been used up, i.e., $C(j,S)=\epsilon$, then recall Rule 2 by complementary slackness, which guides through all the algorithms, $\alpha_j^*$ should be positive, i.e., in set $J_1$. Therefore, for each constraint $j\in J_1$, complementary slackness and previous lemma imply that if $S$ is not $r_j$-bad, then $C(j)\geqslant 1-2\epsilon$. Also for $(i,o)\in O^*$, recall that we allocate agent $i\in I^*$ to $o$ that maximizes $gain(o)$, i.e., $x_{io}>0$. Again by complementary slackness, we have the slack variables in the \textbf{Dual-LP} constraints equal to 0, i.e., $\beta_i+\sum_j \alpha_j^* a_{ioj}=w_{io}$. Now \begin{center}$W= \displaystyle\sum_{{(i,o)}\in O^*}w_{io}=\displaystyle\sum_{{(i,o)}\in O^*}(\beta_i+\displaystyle\sum_j \alpha_j^* a_{ioj})=\displaystyle\sum_{i\in I^*}\beta_i+\displaystyle\sum_j(\alpha_j^* \displaystyle\sum_{(i,o)\in O^*}a_{ioj})=\displaystyle\sum_{i\in I^*}\beta_i+\displaystyle\sum_j \alpha_j^* C(j)= \displaystyle\sum_{i\in I^*}\beta_i+\displaystyle\sum_{j\in J_1} \alpha_j^* C(j)+\displaystyle\sum_{j\in J_2} \alpha_j^* C(j)\geqslant \displaystyle\sum_{i\in I^*}\beta_i+ \displaystyle\sum_{j\in J_1} \alpha_j^* (1-2\epsilon)\geqslant \displaystyle\sum_{i\in I^*}\beta_i (1-2\epsilon)+ \displaystyle\sum_{j\in J_1} \alpha_j^* (1-2\epsilon)=(1-2\epsilon)D.$
\end{center}
Since the options for agents in $S$ were not selected by our algorithm, the total value obtained by the algorithm is $W-W(S)$. Since $S$ is not t-bad, we have by Definition 3, $W(S)\leqslant \epsilon W+ (m+1)(\text{ln} nq)w_{max}+\sqrt{W} (2\sqrt{\epsilon(m+1)(\text{ln} nq)w_{max}})$. The first assumption of Theorem 5 gives us $w_{max}\leqslant OPT \epsilon/ (m+1)(ln (nq))$. We then have $(m+1)(ln nq)w_{max}\leqslant \epsilon OPT$. Hence by substitution \begin{center}$W(S)\leqslant \epsilon W+\epsilon OPT+2\sqrt{W}\sqrt{\epsilon^2 OPT}\leqslant O(\epsilon W)$
\end{center}
Therefore \begin{center}$\sum_{(i,o)\in O^*-O^*(S)}w_{io}=W-W(S)\geqslant W-O(\epsilon W)=W(1-O(\epsilon))\geqslant D(1-O(\epsilon))$
\end{center}
\end{proof}

Now I prove Proposition 7 by first proving another two lemma.
\begin{lemma*}\textbf{Pr}$[S \text{ is } r_j\text{-bad}]\leqslant \frac{1}{m (nq)^{m+1}}$ for each $j$, and \textbf{Pr}$[S \text{ is t-bad}]\leqslant \frac{1}{(nq)^{m+1}}$.
\end{lemma*}
The proof is simple and thus omitted here. The details of the proof could be found in \cite{cliff}. This lemma implies that for any fixed $\alpha^*$, the probability that a random sample $S$ of agents is bad is less than $\frac{2}{(nq)^{m+1}}$. This is because $S$ is bad if any $j$ is $r_j$-bad or t-bad. Therefore $\textbf{Pr}\text{[S is bad]}=\sum_j \textbf{Pr} [S \text{ is } r_j\text{-bad}]+\textbf{Pr}[S \text{ is t-bad}]-\text{some positive intersection probability}\leqslant \frac{1}{(nq)^{m+1}}+\frac{1}{(nq)^{m+1}}=\frac{2}{(nq)^{m+1}}$.
\begin{lemma*}There are fewer than $(nq)^m$ distinct solutions $\alpha^*$ that are returned by the algorithm after step 2 of our algorithm.
\end{lemma*}
\begin{proof}Recall that an optimal (vertex) solution to the \textbf{Dual-LP} on the reduced instance is defined purely by the n-dimensional vector $\alpha^*.$ The vertex solution is decided by picking $m$ constraints, set them to equality, and solve the system. For each of the $n$ agents, there are at most $q$ such constraints, and thus at most $nq$ constraints in total. Consequently, we are picking $m$ of them from at most $nq$ of them, and thus there are at most $nq \choose m$ possible combinations and thus $nq \choose m$ vertices of the polytope defined by optimal solutions $\alpha^*$.
\end{proof}
I now prove Proposition 7:
\begin{proof}The first lemma imples that for any fixed $\alpha^*$, the probability that a random sample $S$ of agents is bad is less than $\frac{2}{(nq)^{m+1}}$. The second lemma tells that there are at most $(nq)^m$ distinct choices for $\alpha^*$. Therefore the probability that $S$ is bad is less than $\frac{2}{nq}$, i.e., the sample is good for any $\alpha^*$ with high probability. We then finish proving Proposition 7.
\end{proof}
We then complete the proof of our central theorem, Theorem 5.\\
\textbf{Remark:} As we can see from the whole proof, the accuracy of the Training-Based Primal-Dual algorithm is bounded by two components: the strictness of the inequaility of the two conditions in Theorem 5, i.e., the smaller the value of the right-hand side of two inequalities, the more likely the algorithm would perform better on the remaning agents; and the accuracy of the $\alpha_i^*$ solved for the sampled dual. As $\epsilon$ increases, the sample we study becomes larger, and $\alpha_i^*$ is more accurate. However, as $\epsilon$ increases, the bound of the $\text{max}_{i,o}\{\frac{w_{io}}{OPT}\}$ and $\text{max}_{i,o}\{\frac{a_{ioj}}{c_j}\}$ becomes less strict, and therefore the performance of the algorithm decreases. In general, the former effects outweights the latter one, as long as the sample we study is not too small, making it $1-O(\epsilon)$ competitive. We need to find a balance between the two effects and choose $\epsilon$ not too large and not too small.

\end{document}